%% file: poincare_arxiv.tex
\documentclass{amsart}

\usepackage{hyperref}
\usepackage{aliascnt}
\usepackage{amscd}

\newcommand{\abs}[1]{\left\vert#1\right\vert}
\newcommand{\norm}[1]{\left\Vert#1\right\Vert}
\newcommand{\D}{\mathcal{D}}
\newcommand{\Po}{\mathcal{P}}
\newcommand{\Q}{\mathcal{Q}}
\newcommand{\R}{\mathcal{R}}
\newcommand{\U}{\mathbb{U}}
\newcommand{\V}{\mathbb{V}}
\newcommand{\W}{\mathbb{W}}
\newcommand{\X}{\mathbb{X}}
\newcommand{\grad}{\nabla}
\DeclareMathOperator{\curl}{curl}
\DeclareMathOperator{\rank}{rank}
\DeclareMathOperator{\Nl}{Null}
\DeclareMathOperator{\im}{image}
\DeclareMathOperator{\Hom}{Hom}
\newcommand{\f}{\varphi}
\newcommand{\LA}{\Lambda}
\newcommand{\OM}{\Omega}
\newcommand{\om}{\omega}
\newcommand{\al}{\alpha}
\newcommand{\g}{\gamma}
\newcommand{\bt}{\beta}
\newcommand{\z}{\zeta}
\newcommand{\s}{\psi}
\newcommand{\la}{\lambda}

\copyrightinfo{2008}{American Mathematical Society}

\newtheorem{theorem}{Theorem}[section]
\newaliascnt{lemma}{theorem}
\newtheorem{lemma}[lemma]{Lemma}
\aliascntresetthe{lemma}
\newaliascnt{ques}{theorem}
\newtheorem{ques}[ques]{Question}
\aliascntresetthe{ques}
\newaliascnt{prop}{theorem}
\newtheorem{prop}[prop]{Proposition}
\aliascntresetthe{prop}

\theoremstyle{definition}
\newtheorem{example}[theorem]{Example}

\numberwithin{equation}{section}

\begin{document}

\title{Elliptic complexes and generalized Poincar\'{e} inequalities}

\author{Derek Gustafson} \address{Department of Mathematics, Syracuse University, Syracuse, NY 13210}
\email{degustaf@syr.edu}

\copyrightinfo{\currentyear} {Derek Gustafson}

\date{\today}

\begin{abstract} \input{abstract.tex} \end{abstract}

\maketitle

\section{Elliptic Complexes}
\input{Elliptic_Complex.tex}

\section{Main Question}
\input{main_question.tex}

\section{Generalized Inverses}
\input{Gen_inv.tex}

\section{Sufficiency of Generalized Inverses}
\input{sufficiency.tex}

\bibliographystyle{amsplain}
\bibliography{poincare_arxiv}

\end{document}

%% file: abstract.tex
We study first order differential operators $\Po = \Po(D)$ with constant coefficients.  The main question is under what conditions a generalized Poincar\'{e} inequality holds $$\norm{ D(f-f_0)}_{L^p} \leq C \norm{ \Po f}_{L^p}, \hspace{.5in} \textrm{for some } f_0 \in \ker \Po.$$ We show that the constant rank condition is sufficient, \autoref{gen_inv}.  The concept of the Moore-Penrose generalized inverse of a matrix comes into play. 

%% file: Elliptic_Complex.tex
Let $\U$, $\V$, and $\W$ be finite dimensional inner product spaces, whose inner products are denoted by $\left<\; , \;\right>_\U$, $\left< \; , \; \right>_\V$, and $\left< \; , \; \right>_\W$ respectively, or just $\left< \; , \; \right>$ when the space is clear. Let $\Po$ and $\Q$ be the first order differential operators with constant coefficients $$\Po= \sum_{i=1}^n A_i \frac{\partial}{\partial x_i}, \hspace{1.2in} \Q= \sum_{i=1}^n B_i \frac{\partial}{\partial x_i},$$ where the $A_i$ are linear operators from $\U$ to $\V$ and the $B_i$ are linear operators from $\V$ to $\W$.  We will use $$\Po(\xi) = \sum_{i=1}^n \xi_i A_i, \hspace{.5in} \textrm{and} \hspace{.5in} \Q(\xi)=\sum_{i=1}^n \xi_i B_i$$ to denote the symbols of $\Po$ and $\Q$, respectively.  We denote by $\D'(\Bbb R^n, \V)$ the space of distributions valued in $\V$.  We define a \textit{short elliptic complex of order 1 over $\Bbb R^n$} to be $$\begin{CD} \D'(\Bbb R^n, \U) @>\Po>> \D'(\Bbb R^n, \V) @>\Q>> \D'(\Bbb R^n, \W) \end{CD}$$ such that the \textit{symbol complex} $$\begin{CD} \U @>\Po(\xi)>> \V @>\Q(\xi)>> \W \end{CD}$$ is exact for all $\xi \neq 0 \in \Bbb R^n$.

There are two classical structures that fit into this framework and provide our motivation for studying elliptic complexes in general.
\begin{example} For $l=0, \dots, n$, let $\LA^l = \LA^l(\Bbb R^n)$ be the space of $l$-covectors on $\Bbb R^n$, that is the vector space with basis elements $dx_{i_1}\wedge \dots \wedge dx_{i_l}$, $1 \leq i_1 < \dots <i_l \leq n$, where all choices of subsets of $\{x_1, \dots, x_n\}$ of cardinality $l$ are used.  By convention, $\LA^0=\Bbb R$ with basis element $1$.  Then we have the Grassmann Algebra $\LA = \LA(\Bbb R^n) = \bigoplus \LA^l(\Bbb R^n)$; that is, the space of  covectors on $\Bbb R^n$.  The relevant elliptic complex is $$\begin{CD} \D'(\Bbb R^n, \LA) @>d>> \D'(\Bbb R^n, \LA) @>d>> \D'(\Bbb R^n, \LA)\end{CD},$$  where $d$ is exterior differentiation which is defined by $$d\left( f dx_{i_1} \wedge \dots \wedge dx_{i_l} \right) = df \wedge dx_{i_1} \wedge \dots \wedge dx_{i_l} = \sum_{j=1}^n \frac {\partial f}{\partial x_j} dx_j \wedge dx_{i_1} \wedge \dots \wedge dx_{i_l}$$ on the basis elements and extended by linearity to $\LA$.  \end{example}

\begin{example} The other classical example, which has led to the recent use of elliptic complexes in the study of PDE's is $$\begin{CD} \D'(\Bbb R^n, \Bbb R) @>\grad>> \D'(\Bbb R^n, \Bbb R^n) @>\curl>> \D'(\Bbb R^n, \Bbb R^{n\times n}_\textrm{skew}) \end{CD}.$$  Here $\Bbb R^{n\times n}_\textrm{skew}$ is the space of $n$ by $n$ skew symmetric matrices and $\curl$ is the rotation operator given by $$\curl\; (f_1, \dots, f_n) = \left[ \frac {\partial f_i} {\partial x_j}  - \frac {\partial f_j} {\partial x_i}  \right]_{i,j}.$$  \end{example}

From an elliptic complex, we form the adjoint complex $$\begin{CD} \D'(\Bbb R^n, \W) @>\Q^*>> \D'(\Bbb R^n, \V) @>\Po^*>> \D'(\Bbb R^n, \U)\end{CD}.$$ Here $\Po^*$ is the formal adjoint defined by $$\int_{\Bbb R^n} \left<\Po^*f,g\right>_{\U} = \int_{\Bbb R^n} \left<f, \Po g\right>_{\V}$$ for $f \in C^\infty_0(\Bbb R^n, \V)$ and $g \in C^\infty_0(\Bbb R^n, \U)$.  So, we have \begin{equation}\label{adj_def}\Po^* = - \sum_{i=1}^n A^*_i \frac{\partial}{\partial x_i},\end{equation} and similarly for $\Q^*$.  Here, we have identified $\U^*$, $\V^*$, and $\W^*$ with $\U$, $\V$ and $\W$, respectively, by use of their inner products.  Note that the adjoint complex is elliptic if and only if the original complex is.

From this, we define an associated second order \textit{Laplace-Beltrami Operator} by $$\triangle = \triangle_{\V} = -\Po \Po^* - \Q^*\Q: \D'(\Bbb R^n, \V) \to \D'(\Bbb R^n, \V),$$ with symbol denoted by $\triangle(\xi):\V \to \V$.  Linear Algebra shows that for every $v\in \V$, $\left< -\triangle(\xi) v, v \right> = \abs{\Po^*(\xi)v}^2 + \abs{\Q(\xi)v}^2 \geq 0$.  That equality only occurs when $\xi=0$ follows from the definition of an elliptic complex.  Thus, the linear operator $\triangle(\xi): \V \to \V$ is invertible for $\xi\neq0$.  We also have that as a function in $\xi$, $\triangle(\xi)$ is homogeneous of degree $2$.  So, letting $$c= \max_{\abs{\xi}=1} \norm{\triangle^{-1}(\xi): \V \to \V}, $$ we get the estimate $$\norm{\triangle^{-1}(\xi)} \leq c \abs{\xi}^{-2}.$$  So, solving the Poisson Equation $$\triangle \f =F$$ with $F \in C_0^\infty(\Bbb R^n, \V)$, we find the second derivatives of $\f$ by noting that $$\widehat{\frac {\partial^2 \f} {\partial x_i \partial x_j}} (\xi) = \xi_i\xi_j \triangle^{-1}(\xi) \widehat{F}(\xi).$$  Since $\xi_i\xi_j \triangle^{-1}(\xi): \V \to \V$ is bounded, this gives rise to a Calder\'{o}n-Zygmund type singular integral operator, $R_{ij}F = \frac {\partial^2} {\partial x_i \partial x_j}\f$ which is bounded on $L^p$ for $1<p<\infty$.  We will refer to these as the second order Riesz type transforms, due to the similarities with the classical Riesz transforms.  A detailed discussion of Calder\'{o}n-Zygmund singular integral operators and, in particular, the classical Riesz transforms can be found in \cite{Stein70}.

We will use $$W^{k,p}(\Bbb R^n)= \left\{f: \sum_{\abs{\al}\leq k} \norm{\frac {\partial^{\abs{\al}}}{\partial x^\al}f}_p <\infty \right\}$$ to denote the classical Sobolev spaces, and $$L^{k,p}(\Bbb R^n) = \left\{f: \norm{f}_{k,p} = \sum_{\abs{\al}=k} \norm{\frac {\partial^{\abs{\al}}}{\partial x^\al}f}_p <\infty \right\} \subset W^{k,p}_{\textrm{loc}}(\Bbb R^n)$$ to denote the space of distributions with all $k^{th}$ order derivatives in $L^p(\Bbb R^n)$.  Note that we can approximate a $L^{k,p}(\Bbb R^n)$ function by $W^{k,p}(\Bbb R^n)$ functions.  This is accomplished by taking $f\in L^{k,p}(\Bbb R^n)$ and multiplying by $\s_n \in C^\infty_0(\Bbb R^n)$ where $0 \leq \s_n \leq 1$, $\s_n=1$ on the ball about $0$ of radius $n$, and has its support contained in the ball about $0$ of radius $2n$.  As $n$ goes to infinity, $f \s_n$ converges to $f$ in the $L^{k,p}$ norm.  So, in this notation, the Poisson Equation $$\triangle \f =F$$ with $F \in L^p(\Bbb R^n)$ is solvable for $\f \in L^{2,p}(\Bbb R^n)$.  Also, since $C^\infty_0(\Bbb R^n, \V)$ is dense in $L^{1,p}(\Bbb R^n, \V)$ and $\Po^*$ and $\Q$ are continuous under the $L^{1,p}$ seminorm, we can extend them by continuity to all of $L^{1,p}(\Bbb R^n, \V)$.  Similarly, $\Po$ and $\Q^*$ can be defined on $L^{1,p}(\Bbb R^n, \U)$ and $L^{1,p}(\Bbb R^n, \W)$, respectively.

We refer the reader to \cite{Donofrio_Iwaniec03}, \cite{Giannetti_Verde00}, \cite{Tarkhanov95}, and \cite{Uhlenbeck77} for further reading on elliptic complexes.   

%% file: main_question.tex
We begin by recalling the classical Poincar\'{e} Inequality
\begin{theorem} For each $ f \in \D'(\Bbb R^n)$ such that $\grad f \in L^p(\Bbb R^n)$and each ball $B\subset \Bbb R^n$, there exists a constant $f_B$ such that $$\int_B \abs{ f- f_B}^p \leq C \int_B \abs{\grad  f}^p.$$  We view $f_B$ as an element in $\D'(\Bbb R^n)$ with $\grad  f_B=0$.
\end{theorem}
This leads to our main question:
\begin{ques}For what partial differential operators $\Po$ of order $k$ is it true that for every $f \in \D'(\Bbb R^n, \U)$ such that $\Po f \in L^p(\Bbb R^n, \V)$, there exists $f_0\in \D'(\Bbb R^n, \U)$ such that $\Po f_0=0$ and \begin{equation}\label{est1}\norm{f-f_0}_{k,p} \leq C \norm{\Po f}_p?\end{equation}
\end{ques}
Notice that with the change from $\grad$ to $\Po$ we also had to change some other details.  First, there is no need for the ball that appears in the classical theorem, our methods have been able to achieve global estimates.  But, our estimates are on the $k^{th}$ order partial derivatives of $f$, not $f$ itself.  The local $L^p$ estimates of $f-f_0$ will follow from \autoref{est1} by the usual Poincar\'{e} inequality.  We will confine our investigations to the case $k=1$.


We present two theorems as partial answers to this question.  The first is already known, see \cite{Giannetti_Verde00} for example,and uses the methods of elliptic complexes to attack the problem. The second theorem takes a more direct approach which allows for a more general result.

\begin{theorem}\label{Elliptic_complex}
Let $1<p<\infty$, and let $$\begin{CD} \D'(\Bbb R^n, \X) @>\R>> \D'(\Bbb R^n, \U) @>\Po>> \D'(\Bbb R^n, \V) @>\Q>> \D'(\Bbb R^n, \W)\end{CD}$$ be an elliptic complex of order 1, and let $f \in \D'(\Bbb R^n, \U)$ such that $\Po f\in L^p(\Bbb R^n, \V)$.  Then there exists $f_0 \in \D'(\Bbb R^n, \U) \cap \ker \Po$ with $$\norm{f-f_0}_{1,p} \leq C \norm{\Po f}_p.$$
\end{theorem}

\begin{proof}
Here we shall need not only the Laplace-Beltrami Operator for functions valued in $\V$, but also the Laplace-Beltrami Operator for functions valued in $\U$, $\triangle_\U = \R\R^* + \Po^*\Po$.  There exists $\f \in \D'(\Bbb R^n, \U)$ such that $\triangle_\U \f =f$.  Note that because of the exactness of the elliptic complex, we have the identity $$\triangle_\V \Po\f = \Po\Po^*\Po\f + \Q^*\Q\Po\f = \Po\Po^*\Po\f + \Po\R\R^*\f = \Po\triangle_\U\f = \Po f.$$  Let $f_0 = f - \Po^*\Po\f$.  Now it simply remains to verify that $f_0$ satisfies the conclusions of the theorem.  First, $$\Po f_0 = \Po f - \Po\Po^*\Po\f = \Po f - \Po\Po^*\Po\f - \Po\Q\Q^*\f = \Po f - \Po \triangle_\U \f =0.$$  Also, \begin{eqnarray*}\sum_i \norm{\frac {\partial}{\partial x_i} (f-f_0)}_p & = & \sum_i \norm{\frac {\partial}{\partial x_i} \Po^*\Po\f}_p \leq \sum_{i,j} \norm{A_j^* \frac {\partial^2}{\partial x_i \partial x_j} \Po\f}_p \\ & \leq & \sum_{i,j} \norm{ A_j^* R_{ij} \Po f}_p \leq \sum_{i,j} \norm{A_j^*} C_{i,j} \norm{\Po f}_p \\ & \leq & C\norm{\Po f}_p.\end{eqnarray*}
\end{proof}

%% file: Gen_inv.tex
Before we are able to present the second theorem, we need to look at the theory of generalized inverses.
\begin{prop}For $A\in \Hom(\Bbb U, \Bbb V)$, there exists a unique $A^\dagger \in \Hom(\Bbb V, \Bbb U)$, called the \textit{Moore-Penrose generalized inverse},with the following properties:
\begin{enumerate}
\item $AA^\dagger A=A: \U \to \V$,
\item $A^\dagger AA^\dagger = A^\dagger: \V \to \U$,
\item $(AA^\dagger)^*=AA^\dagger: \V \to \V$,
\item $(A^\dagger A)^* = A^\dagger A: \U \to \U$.
\end{enumerate}\end{prop}

The linear map $A^\dagger$ has properties similar to inverse matrices that make it valuable as a tool.
\begin{prop}For $\la\neq 0$, $(\la A)^\dagger = \la^{-1} A^\dagger$.\end{prop}
\begin{prop}For a continuous matrix valued function $P=P(\xi)$, the function $P^\dagger = P^\dagger(\xi)$ is continuous at $\xi$ if and only if there is a neighborhood of $\xi$ on which $P$ has constant rank.\end{prop}
\begin{prop}$AA^\dagger$ is the orthogonal projection onto the image of $A$.  $A^\dagger A$ is the orthogonal projection onto the orthogonal complement of the kernel of $A$.\end{prop}

For a more detailed discussion of generalized inverses and the proofs of these results, consult \cite{Campbell_Meyer79} and the references cited there.
Generalized Inverses are the additional tools we need for the following theorem.
\begin{theorem}\label{gen_inv}Let $\Po:\D'(\Bbb R^n, \U) \to \D'(\Bbb R^n, \V)$ be a differential operator of order 1 with constant coefficients and symbol $\Po(\xi)$ which is of constant rank for $\xi\neq 0$, and let $f \in \D'(\Bbb R^n, \U)$ such that $\Po f\in L^p(\Bbb R^n, \V)$, $1<p<\infty$.  Then there exists $f_0 \in \D'(\Bbb R^n, \U)$ such that $\Po f_0=0$ and $$\norm{f-f_0}_{1,p} \leq C \norm{\Po f}_p.$$
\end{theorem}

Note that this is the same constant rank condition investigated in \cite{FonsecaMuller99} in relation to quasiconvexity of variational integrals.


\begin{proof}
From the symbol $\Po(\xi): \U \to \V$, we have its generalized inverse $\Po^\dagger (\xi): \V \to \U$.  We use this to define a pseudodifferential operator $R_j$, which we will refer to as the first order Riesz type transforms.  For $h \in C_0^\infty(\Bbb R^n, \V)$, we define $R_j h (x) = (2\pi)^{-n/2} \int i e^{ix\cdot \xi} \xi_j\Po^\dagger (i\xi) \widehat{h}(\xi) d\xi$. Note that since $\Po(\xi)$ is homogeneous of degree 1, we get that for $\la\neq 0$ $$(\la\xi_j)\Po^\dagger(i\la\xi) = \la\xi_j \left( \la \Po(i\xi)\right)^\dagger = \xi_j\Po(i\xi).$$  So, $\xi_j\Po^\dagger(i\xi)$ is homogeneous of degree 0.  Since $\Po(\xi)$ has constant rank away from the origin, $\Po^\dagger(i\xi)$ is continuous on $\abs{\xi}=1$.  Thus, $\xi_j\Po^\dagger(i\xi)$ is continuous and homogeneous of order 0, which makes it bounded.  Therefore, $R_j$ extends continuously to an operator from $L^p(\Bbb R^n, \V)$ to $L^p(\Bbb R^n, \U)$.  Recalling the definition of the operator $\Po = \sum_j A_j \frac {\partial} {\partial x_j}$, we note that $$\sum_j A_j R_j h = (2\pi)^{-n/2} \int e^{ix\cdot \xi} \Po(i \xi) \Po^\dagger(i\xi) \widehat{h} (\xi) d\xi.$$  We also note that $\Po(i\xi)\Po^\dagger(i\xi)$ is the orthogonal projection onto the image of $\Po(i\xi)$.  In particular this means that if $h=\Po g$, then $\sum_j A_j R_j h = h$, and since this is defined by a Calder\'{o}n-Zygmund singular integral operator this extends to all of $L^p(\Bbb R^n, \V)$.

The reader may wish to notice that $$\frac {\partial} {\partial x_j} R_k h = \frac {\partial} {\partial x_k} R_j h.$$  Therefore, there exists a distribution $f_0$ such that $\frac {\partial}{\partial x_j} f_0 = \frac {\partial}{\partial x_j} f - R_j \Po f$, for $j=1, \dots, n$.  Then, we have $$\Po f_0 = \Po f - \sum_j A_j R_j \Po f = \Po f - \Po f =0.$$  And, $$\norm{f-f_0}_{1,p} \leq \sum_j \norm{R_j \Po f}_p \leq C \norm{\Po f}_p,$$ where the constant depends on the norms of the Riesz type transforms.
\end{proof}

%% file: sufficiency.tex
At this point we have proved \autoref{Elliptic_complex} and \autoref{gen_inv} in an attempt to answer our question about when a generalized Poincar\'{e} inequality is true.  What is unclear is if these two results are related in any way.  Before we can answer this question, we will need the following lemma.  It should be noted that although it is elementary in nature, we were unable to find it in the literature.

\begin{lemma} \label{lsc}Let $\U$ and $\V$ be finite dimensional innerproduct spaces.  Let $r:\Hom(\U, \V) \to \Bbb N_0$  be the function that takes a matrix to its rank.  If $\Hom(\U, \V)$ is given the operator norm, then $r$ is lower semicontinuous.
\end{lemma}

\begin{proof}
Since $r$ only takes values in the nonnegative integers, it is enough to show that for each natural number $t$, the set $\{r(A) \leq t-1\}$ is closed.  Let $u_1, \dots, u_l$ be an orthonormal basis for $\U$.  Let $\OM$ be the collection of ordered $t$-tuples $\om= (i_1, \dots, i_t)$ where $1 \leq i_1 < \dots < i_t \leq l$.  For each $\om \in \OM$, we define a function $\Phi_\om: \Hom(\U, \V) \times \Bbb{S}^{t-1} \to \V$ by $$\Phi_\om ( A, \mathbf{m}) = \sum_{\nu=1}^t m_\nu A u_{i_\nu}, \hspace{.5in} \left( \sum_{\nu=1}^t m_\nu^2 =1 \right)$$ where $\mathbf{m}=(m_1, \dots, m_t) \in \Bbb{S}^{t-1}$.  Consider $\Phi_\om^{-1} (0)$.  This is a closed set in $\Hom(\U, \V) \times \Bbb{S}^{t-1}$.  Denote the projection of this set onto $\Hom(\U, \V)$ by $Z_\om$.  Note that $Z_\om$ is closed since $\Bbb{S}^{t-1}$ is compact.  Now, let $$Z = \bigcap_{\om\in\OM} Z_\om.$$  Note that $Z$ is a closed set.  Also, $Z$ is precisely the collection of matrices of rank less than $t$.
\end{proof}



Since this next result is true for a broader class of elliptic complexes than what we have previously defined, we will take a moment for definitions so that we may state our result in this broader sense.  For differential operators $$\Po = \sum_{\abs\al \leq m} A_{\al}(x) D^\al$$ and $$\Q = \sum_{\abs\al \leq m} B_{\al}(x) D^\al$$ of order $m$ with variable coefficients, then $$\begin{CD} \D'(\Bbb R^n, \U) @>\Po>> \D'(\Bbb R^n, \V) @>\Q>> \D'(\Bbb R^n, \W) \end{CD}$$ is an elliptic complex of order $m$ if $\Q\Po=0$ and the symbol complex $$\begin{CD} \U @>\Po_m(x,\xi)>> \V @>\Q_m(x,\xi)>> \W \end{CD}$$ is exact for every $x$ and every $\xi \neq 0$.  Here, $\Po_m$ denotes the principle symbol of $\Po$, that is $\sum_{\abs\al=m} A_\al(x) \xi^\al$, and similarly for $\Q_m$.

\begin{theorem}
A sequence $$\begin{CD} \D'(\Bbb R^n, \U) @>\Po>> \D'(\Bbb R^n, \V) @>\Q>> \D'(\Bbb R^n, \W)\end{CD}$$ with continuous coefficients is an elliptic complex if and only if all of the following hold:
\renewcommand{\theenumi}{\roman{enumi}}
\begin{enumerate}
\item $\Q\Po=0$.
\item The sequence $$\begin{CD} \U @>\Po_m(y,\z)>> \V @>\Q_m(y,\z)>> \W \end{CD}$$ is exact for some $y$ and some $\z\neq0$.
\item For each multi-index $\g$ of length $2m$, $$\sum_{\substack{\al+\bt=\g, \\ \abs\al = \abs\bt =m}} B_\bt(x) A_\al(x) =0$$ as operators from $\U$ to $\V$.
\item The matrix $\Po_m(x,\xi)$ has constant rank for all $x$ and all $\xi\neq0$.
\item The matrix $\Q_m(x,\xi)$ has constant rank for all $x$ and all $\xi\neq0$.
\end{enumerate}
\end{theorem}

\begin{proof}
We will begin by showing that an elliptic complex has the stated properties.  (i) and (ii) follow trivially from the definition.  Note that \begin{eqnarray}\label{EC} \Q_m (x,\xi) \Po_m (x,\xi) & = & \sum_{\abs\bt=m} \sum_{\abs\al=m} B_\bt(x) \xi^\bt A_\al(x) \xi^\al \\ \nonumber & = & \sum_{\abs\g =2m} \xi^\g \sum_{\substack{\al+\bt=\g \\ \abs\al = \abs\bt =m}} B_\bt(x) A_\al(x).\end{eqnarray}  Since this is true as functions of $\xi$ and $\Q_m(x,\xi) \Po_m(x,\xi)=0$, we get (iii) be equating coefficients of $\xi^\g$.  Since the $A_\al(x)$ and $B_bt(x)$ are continuous, we get that $\Po_m$ and $\Q_m$ are continuous, so $\rank \Po_m$ and $\rank \Q_m$ are lower semicontinuous by \autoref{lsc}.  By the rank-nullity theorem and the fact that the symbol complex is exact, we get \begin{equation}\label{r-n} \rank \Q_m(x,\xi) = \dim \V - \Nl\Q_m(x,\xi) = \dim\V - \rank\Po_m(x,\xi)\end{equation} as long as $\xi\neq 0$.  Thus, $\rank \Q_m$ is upper semicontinuous since $\rank \Po_m$ is lower semicontinuous away from $\xi=0$.  Therefore it is continuous.  And, since it is valued in a discrete set, we get (v).  Then, (iv) follows by \autoref{r-n}.

Now, we will assume that properties (i) through (iv) hold and show that the complex is elliptic.  \autoref{EC} and (iii) give us that the composition $\Q_m\Po_m$ is identically $0$, which means that $\im\Po_m \subseteq \ker \Q_m$.  Now, consider \begin{eqnarray*}\rank \Po_m(x,\xi) & = & \rank \Po_m (y,\z) = \Nl\Q_m(y,\z) \\ & = & \dim \V - \rank\Q_m(y,\z) = \dim\V - \rank\Q_m(x,\xi) \\ & = & \Nl\Q_m(x,\xi).\end{eqnarray*}  So, as long as $\xi\neq0$ so that we have the constant rank necessary for the first and fourth equality, we have exactness.
\end{proof}